\documentclass[12pt]{article}
\usepackage{graphicx}
\usepackage{tikz}
\usepackage[margin=1in]{geometry}
\usepackage{amsmath, amssymb, amsfonts, amsthm}
\usepackage{mathrsfs}
\usepackage{amsfonts}
\usepackage{amsmath,amssymb,amsfonts,amsthm,fancyhdr,mathrsfs}
\usepackage{bbm}
\usepackage{verbatim}
\usepackage{color}
\usepackage{dsfont}
\usepackage{latexsym}
\usepackage{graphicx}
\usepackage{enumerate}
\newtheorem{lem}{\noindent {\bf Lemma}}[section]

\newtheorem{thm}{\noindent {\bf Theorem}}[section]
\newcounter{remark}
\newcounter{example}
\catcode`@=11 
\date{}
\newcounter{defi}
\newenvironment{defi}{
\smallskip \noindent
{\bf
  Definition \arabic{section}.\arabic{defi}.
}}{\addtocounter{defi}{1}\par}

\title{\bf Volterra-composition operators on the weighted Bergman space with exponential type weights}

\author{  Pan Xiaohua$^1$,\,\,\, Wu Yan$^2$,\,\,\, Xu Xianmin$^2$ \\
{\small 1. Department of Mathematics, Zhejiang Normal University,}\\
{\small Jinhua, Zhejiang 321004, P. R. China}\\
{\small 2. College of Mathematics Physics and Information Engineering, }\\
{\small Jiaxing University, Jiaxing, Zhejiang 314001, P. R. China}
}
\begin{document}
\maketitle
\begin{center}
\begin{minipage}{0.9\textwidth}
\noindent{\bf Abstract.}  The properties of Volterra-composition operators on the weighted Bergman space with exponential type weights  are investigated in this paper. For a certain class of subharmonic function~$\psi:\mathbb{D}\rightarrow\mathbb{R}$, we state some necessary and sufficient conditions that a Volterra-composition operator $V_{\varphi}^{g}$ from the weighted Bergman space~$AL_{\psi}^{2}(\mathbb{D})$ to Bloch type space~$B_{\psi}(\mathbb{D})$ ( or little Bloch type space~$B_{\psi,0}(\mathbb{D})$) must satisfy for $V_{\varphi}^{g}$ to be bounded or compact.

{\bf MR(2010) Subject Classification} 74D05; 35Q72; 35L70; 35B40
\end{minipage}
\end{center}
\footnote{
This research was supported by
the National Natural Science Foundation of China under Grant (No.11301224,11401256,11501249)

}
\begin{section}{Introduction}
Let ~$\mathbb{D}=\{z:~|z|<1\}$~be the open unit disk in the complex plane  $\mathbb{C}$ and $dA$  be the normalized area measure on $\mathbb{D}$ .
 Let $L^{2}(\mathbb{D},dA)$ be the space of square integrable functions and let $L^{\infty}(\mathbb{D},dA)$ be the space of essential bounded measurable functions. We will use abbreviations $L^{2}(\mathbb{D})$ for $L^{2}(\mathbb{D},dA)$ and $L^{\infty}(\mathbb{D})$ for $L^{\infty}(\mathbb{D},dA)$.

 For a subharmonic function $\psi:\mathbb{D}\rightarrow\mathbb{R}$, let $L^{\infty}_{\psi}(\mathbb{D})$ be the space of measurable functions~$f$~on~$\mathbb{D}$~ such that ~$e^{-\psi}f\in L^{\infty}(\mathbb{D})$
 and let $L^{2}_{\psi}(\mathbb{D})$ be the Hilbert space of measurable function ~$f$~on~$\mathbb{D}$~ such that
 $$\|f\|_{L_{\psi}^{2}}=(\int_{\mathbb{D}}|f|^{2}e^{-2\psi}dA)^{\frac{1}{2}}<\infty$$
 The inner product of $L^{2}_{\psi}(\mathbb{D})$ is given by
 $$\langle f, g\rangle_{L_{\psi}^{2}}=\int_{\mathbb{D}}f\overline{g}e^{-2\psi}dA$$

 Let~$H^\infty_\psi(\mathbb{D})$~be the subspace of ~$L^\infty_\psi(\mathbb{D})$~consisting of analytic functions and ~$AL^{2}_{\psi}(\mathbb{D})$~be the closed subspace of $L^{2}_{\psi}(\mathbb{D})$ consisting of analytic functions.
 ~$AL^{2}_{\psi}(\mathbb{D})$~ is called the weighted Bergman space with exponential type weights(see [3],[4]).

 An analytic function ~$f$~is said to belong to the  Bloch-type space ~$B_{\psi}(\mathbb{D})$~if

 $$b_{\psi}(f)=\sup_{z\in\mathbb{D}}e^{-2\psi(z)}|f'(z)|<\infty.$$
 Let ~$B_{\psi,0}(\mathbb{D})$~denote the subspace of ~$B_{\psi}(\mathbb{D})$~such that
 $$\lim_{|z| \rightarrow 1}e^{-2\psi(z)}|f'(z)|=0.$$
 This space is called little Bloch-type space.

\begin{defi}$^{[1],[2]}$\label{TH1}~~For real valued function $\psi\in C^{2}(\mathbb{D})$ with $\Delta\psi>0$,
where $\Delta$ is the Laplace operator.
Let $\tau(z)=(\Delta\psi(z))^{-\frac{1}{2}}$. We say that $\psi\in\mathcal{D}$ if the following conditions are satisfied.\\
~~$($1$)$~There exists a constant $C_{1}>0$ such that $|\tau(z)-\tau(\xi)|\leq C_{1}|z-\xi|$ for $z,\xi\in\mathbb{D}$.\\
~~$($2$)$~There exists a constant $C_{2}>0$ such that $\tau(z)\leq C_{2}(1-|z|)$ for $z\in\mathbb{D}$.\\
~~$($3$)$~There exist constants $0<C_{3}<1$ and $a>0$ such that $\tau(z)\leq \tau(\xi)+C_{3}(1-|z|)$ for  $|z-\xi|>a\tau(\xi)$.
\end{defi}

Let $H(\mathbb{D})$ be the space of analytic functions on $\mathbb{D}$ and suppose that $g:\mathbb{D}\rightarrow\mathbb{C}$ is an analytic function, $\varphi$ is an analytic self-map of the unit disk and $f\in H(\mathbb{D})$,
the Riemann-Stieltjes operator is defined by
$$T_{g}f(z)=\int_{0}^{1}f(tz)zg'(tz)dt=\int_{0}^{z}f(t)g'(t)dt,~~~~~z\in\mathbb{D}.$$
The composition operator $C_{\varphi}$ is defined as $C_{\varphi}f\triangleq f\circ \varphi$.

Let $V_{\varphi}^{g}$ be the  Volterra-composition operator which is defined as
$$V_{\varphi}^{g}f(z)\triangleq (T_{g}\circ C_{\varphi}f)(z)=\int_{0}^{z}f(\varphi(t))g'(t)dt,\quad\forall f\in H(\mathbb{D}), ~z\in\mathbb{D}.$$
Note that when $\varphi$ is an analytic self-map of the unit disk and $f,g\in H(\mathbb{D})$, we have $$(V_{\varphi}^{g}f)^{'}(z)=f(\varphi(z))g^{'}(z).$$

In this paper, we will characterize the boundedness and compactness of the operator $V_{\varphi}^{g}$ from the weighted Bergman space with exponential type weights to the Bloch-type space (or little Bloch-type space). These properties of Hankel operators have been considered(see [1],[2]).

We assume that $H_{\psi}^{\infty}(\mathbb{D})$ is dense in $AL_{\psi}^{2}(\mathbb{D})$,  the set of all polynomials is dense in $AL_{\psi}^{2}(\mathbb{D})$.

\end{section}

\begin{section}{Preliminaries}

In order to prove the main results, we need the following lemmas.

\begin{lem}
Let $$\|f\|_{B_{\psi}}=|f(0)|+\sup_{z\in\mathbb{D}}e^{-2\psi(z)}|f^{'}(z)|,$$ then Bloch--type space $B_{\psi}(\mathbb{D})$ is a Banach space with the norm $\|\cdot\|_{B_{\psi}}$.
\end{lem}
\par
\begin{proof}
At first, we will show that $\|\cdot\|_{B_{\psi}}$ is a norm.
For a function $f\in B_{\psi}(\mathbb{D})$, we define
$$b_{\psi}( f)=\sup_{z\in\mathbb{D}}e^{-2\psi(z)}|f^{'}(z)|.$$
It is easy to see that
$$\|\alpha f\|_{B_{\psi}}=|\alpha f(0)|+b_{\psi}(\alpha f)=|\alpha||f(0)|+|\alpha|b_{\psi}(f)=|\alpha|\| f\|_{B_{\psi}}.$$
\begin{eqnarray*}
\|f+g\|_{B_{\psi}}&=&|f(0)+g(0)|+b_{\psi}(f+g)
\leq|f(0)|+b_{\psi}(f)+|g(0)|+b_{\psi}(g)\\
&=&\| f\|_{B_{\psi}}+\| g\|_{B_{\psi}}.
\end{eqnarray*}
So $\|\cdot\|_{B_{\psi}}$ ~is a semi-norm.\\

Assume that $\|f\|_{B_{\psi}}=|f(0)|+b_{\psi}(f)=0$, we have $f(0)=0$ and $b_{\psi}(f)=0$, which means that
$$\sup_{z\in\mathbb{D}}e^{-2\psi(z)}|f^{'}(z)|=0.$$
 Since $e^{-2\psi(z)}\neq0$ when $z\in\mathbb{D}$, it must be $|f'(z)|\equiv 0$. It follows that $f(z)\equiv C $ for some constant $C$. Since $f(0)=0$, we obtain $C=0$.  It implies that $f=0$.

Now we are going to prove the completeness of $\|\cdot\|_{B_{\psi}}$.
Suppose that~$\{f_{n}\}_{n\in\mathbb{N}}$ is a Cauchy sequence of $B_{\psi}(\mathbb{D})$, i.e. $$\|f_{n}-f_{m}\|_{B_{\psi}}=|f_{n}(0)-f_{m}(0)|+\sup_{z\in\mathbb{D}}e^{-2\psi(z)}|f_{n}^{'}(z)-f_{m}^{'}(z)|\rightarrow 0 ~~~~(n,m\rightarrow\infty).$$

Since
\begin{eqnarray*}
|f_{n}(z)-f_{m}(z)|&=&|\int_{0}^{z}[f_{n}^{'}(t)-f_{m}^{'}(t)]dt+[f_{n}(0)-f_{m}(0)]|\\
&\leq&\int_{0}^{z}|f_{n}^{'}(t)-f_{m}^{'}(t)|dt+|f_{n}(0)-f_{m}(0)|,
\end{eqnarray*}
for each $z\in \mathbb{D}$ , there exists a $f\in H(\mathbb{D})$ such that
$f_{n}$ uniformly converges to $f$ on any compact subsets of ~$\mathbb{D}$~.
By Cauchy integral formula, we have $f^{'}_n$~uniformly converges to ~$f^{'}$~on any compact subsets of ~$\mathbb{D}$~.

Since $b_{\psi}(f_{n}-f_{m})\rightarrow 0$,  for any $\varepsilon>0$, there exists a positive integer $N$, such that when $n,m>N$, we have
$$e^{-2\psi(z)}|f_{n}^{'}(z)-f_{m}^{'}(z)|<\varepsilon,~~~~for ~all ~z\in\mathbb{D}.$$
Letting $m\rightarrow\infty$ in the above inequality, we can see that
$$e^{-2\psi(z)}|f_{n}^{'}(z)-f^{'}(z)|\leq\varepsilon,~~~~for~all~ z\in \mathbb{D},$$
this implies that $b_{\psi}(f_{n}-f)\rightarrow 0~~(n\rightarrow \infty)$.
It follows that $\|f_{n}-f\|_{B_{\psi}}\rightarrow 0~(n\rightarrow\infty).$
It is easy to see that there exists a positive integer $n$ such that for all ~$z\in \mathbb{D}$~ $$e^{-2\psi(z)}|f^{'}(z)|\leq1+e^{-2\psi(z)}|f_{n}^{'}(z)|\leq1+b_{\psi}(f_{n})<\infty.$$
i.e. $f\in B_{\psi}$, thus the completeness is proved.

Therefore, $B_{\psi}(\mathbb{D})$ is a Banach space with the norm $\|\cdot\|_{B_{\psi}}$.
\end{proof}

{\bf Notation:} we write $f(z,w)\sim g(z,w)$ if there exist positive constants $C$ and $C'$ such that $$C g(z,w)\leq f(z,w)\leq C' g(z,w)$$.

\begin{lem}$^{[2]}$
Let $\psi\in\mathcal{D}$, then we have $$K(z,z)e^{-2\psi(z)}\sim(\tau(z))^{-2}=\Delta\psi(z),~~z\in\mathbb{D}.$$
where $K(z,w)$ is the reproducing kernel of $AL_{\psi}^{2}(\mathbb{D})$.
\end{lem}

\begin{lem}
Let $f\in AL_{\psi}^{2}(\mathbb{D})$ and $\psi\in\mathcal{D}$, then $$|f(z)|\leq\sqrt{K(z,z)}\|f\|_{L_{\psi}^{2}}.$$
\end{lem}
\par
\begin{proof}
$|f(z)|=|\langle f,K_{z}\rangle|\leq\|f\|_{L_{\psi}^{2}}\|K_{z}\|_{L_{\psi}^{2}}=\sqrt{K(z,z)}\|f\|_{L_{\psi}^{2}}$.
\end{proof}

\begin{lem}
Let $\varphi$ be an analytic self-map of the unit disk $\mathbb{D}$ and $g\in H(\mathbb{D})$. Assume that
$V_{\varphi}^{g}:AL_{\psi}^{2}(\mathbb{D})\rightarrow B_{\psi}(\mathbb{D})$ is bounded. Then $V_{\varphi}^{g}$ is compact if and only if $V_{\varphi}^{g}f_{k}$ converges to zero as $k\rightarrow\infty$, for any bounded sequence $\{f_{k}\}_{k\in\mathbb{N}}$ in $AL_{\psi}^{2}(\mathbb{D})$ which uniformly converges to zero on compact subsets of $\mathbb{D}$.
\end{lem}
\par
\begin{proof}

First, assume that $V_{\varphi}^{g}$ is compact. Let $\{f_{k}\}_{k\in\mathbb{N}}$ be any bounded sequence in $AL_{\psi}^{2}(\mathbb{D})$ and uniformly converges to zero on compact subsets of $\mathbb{D}$. Since $V_{\varphi}^{g}$ is a compact operator , there exists a subsequence of $\{f_{k}\}_{k\in\mathbb{N}}$~(~without loss of generality, we assume it is $\{f_{k}\}_{k\in\mathbb{N}}$~)~ and a function $h\in B_{\psi}(\mathbb{D})$ such that$$\|V_{\varphi}^{g}f_{k}-h\|_{B_{\psi}}\rightarrow0 (k\rightarrow\infty).$$ That is, $$|V_{\varphi}^{g}f_{k}(0)-h(0)|+\underset {z\in\mathbb{D}}{\textup{sup}}e^{-2\psi(z)}|f_{k}(\varphi(z))g^{'}(z)-h^{'}(z)|\rightarrow0(k\rightarrow\infty).$$
By the definition of $V_{\varphi}^{g}$, it is obvious that $V_{\varphi}^{g}f_{k}(0)=0$. It follows that $h(0)=0.$
It is not difficult to see that $f_{k}(\varphi(z))g^{'}(z)$ uniformly converges to $h^{'}(z)$ on any compact subsets of $\mathbb{D}$.
Besides, $\{f_{k}\}_{k\in\mathbb{N}}$ uniformly converges to zero on any compact subsets of $\mathbb{D}$ and $g\in H(\mathbb{D})$,
it follows that $h^{'}(z)\equiv0.$ Also, as $h(0)=0$, we have $h(z)\equiv0$, that is $$\|V_{\varphi}^{g}f_{k}\|_{B_{\psi}}\rightarrow0 (k\rightarrow\infty).$$

Conversely, let $\{f_{k}\}_{k\in\mathbb{N}}$ be any bounded sequence in $AL_{\psi}^{2}(\mathbb{D})$. By Lemma 2.3, we have $|f_{k}(z)|\leq\sqrt{K(z,z)}\|f_{k}\|_{L_{\psi}^{2}}.$
therefore, $\{f_{k}\}_{k\in\mathbb{N}}$ is uniformly bounded on any compact subset $K$
of $\mathbb{D}$.

By Montel Theorem, there exists a subsequence $\{f_{k_{n}}\}_{k_n\in\mathbb{N}}$ of $\{f_{k}\}_{k\in\mathbb{N}}$ and an analytic function $f$, such that $\{f_{k_{n}}\}$ converges to
$f$ uniformly on any compact subsets $K$ of the unit disk $\mathbb{D}$.

According to the Fatou Lemma, $f\in AL_{\psi}^{2}(\mathbb{D})$. Together with the assumption, we have $\|V_{\varphi}^{g}(f_{k_{n}}-f)\|_{B_{\psi}}\rightarrow0 ~(n\rightarrow\infty).$  Therefore,
$\{V_{\varphi}^{g}f_{k}\}_{k\in\mathbb{N}}$ has subsequence $\{V_{\varphi}^{g}f_{k_{n}}\}_{n\in\mathbb{N}}$ which converges to $V_{\varphi}^{g}f\in AL_{\psi}^2(\mathbb{D})$. From the characterization of compact operators, we see that $V_{\varphi}^{g}:AL_{\psi}^{2}(\mathbb{D})\rightarrow B_{\psi}(\mathbb{D})$ is compact.
\end{proof}

\begin{lem}
Let $\psi\in\mathcal{D}$, then the bounded closed set $K$ of $B_{\psi,0}(\mathbb{D})$ is compact if and only if $\underset {|z|\rightarrow1}{\textup{lim}}\underset {f\in K}{\textup{sup}}e^{-2\psi(z)}|f^{'}(z)|=0.$
\end{lem}
\begin{proof}
Assume that the bounded closed set $K$ of $B_{\psi,0}(\mathbb{D})$ is compact. Let $\varepsilon>0$, we have $\underset{f\in K}{\bigcup}B(f,\frac{\varepsilon}{2})\supset K$, where $B(f,\frac{\varepsilon}{2})$ is a ball with center $f$ and radius $\frac{\varepsilon}{2}$.
Because $K$ is compact, there exist $f_{1},f_{2},\cdot\cdot\cdot,f_{n}\in K$, such that for any $f\in K$, we have
$$\|f-f_{j}\|_{B_{\psi}}<\frac{\varepsilon}{2} ~~~~(1\leq j\leq n,j\in\mathbb{N}^{*}).$$
Hence, for any $z\in\mathbb{D}$, we have, for $1\leq j\leq n$,
$$e^{-2\psi(z)}|f^{'}(z)|\leq e^{-2\psi(z)}|f_{j}^{'}(z)|+\frac{\varepsilon}{2}.$$

 Since $f_{j}\in B_{\psi,0}(\mathbb{D})$, for each positive integer $j$, there exists a positive real number $r_{j}\in(0,1)$ such that $e^{-2\psi(z)}|f_{j}^{'}(z)|<\frac{\varepsilon}{2}$ when $r_{j}<|z|<1$. Let $r=\textup{max}\{r_{1},r_{2},\cdot\cdot\cdot,r_{n}\}$, then we have, when $r<|z|<1$,
 $$e^{-2\psi(z)}|f^{'}(z)|\leq\varepsilon.$$
 Namely, $$\underset {|z|\rightarrow1}{\textup{lim}}\underset {f\in K}{\textup{sup}}e^{-2\psi(z)}|f^{'}(z)|=0.$$

Conversely, suppose that the sequence $\{f_{n}\}_{n\in\mathbb{N}}\subset K$. Since $K$ is bounded, there is a positive constant $M$ such that for any $n\in \mathbb{Z^{+}} $, $\|f_{n}\|_{B_{\psi}}\leq M$.
By Lemma 2.3, it is easy to see $\{f_{n}\}_{n\in\mathbb{N}}$ is uniformly bounded on any compact subsets of $\mathbb{D}$.
By Montel Theorem, there exists subsequence of $\{f_{n}\}_{n\in\mathbb{N}}$ ~(~without loss of generality , we suppose it to be $\{f_{n}\}_{n\in\mathbb{N}}$~)~and an analytic function $f$, such that $\{f_{n}\}_{n\in\mathbb{N}}$ converges to $f$ uniformly on any compact subsets of $\mathbb{D}$.

Since
$$\underset {|z|\rightarrow1}{\textup{lim}}\underset {f\in K}{\textup{sup}}e^{-2\psi(z)}|f^{'}(z)|=0$$ and $\{f_{n}\}_{n\in\mathbb{N}}\subseteq K$, for any $\varepsilon>0$,  there exists an $r\in(0,1)$, such that $e^{-2\psi(z)}|f_{n}^{'}(z)|<\frac{\varepsilon}{2}$ for all $n$, when $r<|z|<1$.
Furthermore, $e^{-2\psi(z)}|f^{'}(z)|<\frac{\varepsilon}{2}$ when $r<|z|<1$, therefore we have
$$e^{-2\psi(z)}|f_{n}^{'}(z)-f^{'}(z)|<\varepsilon$$ when $r<|z|<1$.

 Combines with Cauchy Integral Theorem, $f_{n}^{'}$ converges to $f^{'}$ uniformly on any compact subset of $\mathbb{D}$. Note that $e^{-2\psi(z)}$ is continuous on $r\overline{\mathbb{D}}$, hence for any $z\in r\overline{\mathbb{D}}$, we have $|e^{-2\psi(z)}|\leq C$ for some constant $C$. As $f_{n}^{'}$ converges to $f^{'}$ uniformly on $r\overline{\mathbb{D}}$, there exists an integer $G$, when $n\geq G$, $|f_{n}^{'}(z)-f^{'}(z)|<\frac{\varepsilon}{C}$ holds for any $z\in r\overline{\mathbb{D}}$.

Thus, it is easy to see, when $n\geq G$, that $e^{-2\psi(z)}|f_{n}^{'}(z)-f^{'}(z)|<\varepsilon$ holds for any $z\in \mathbb{D}$. This implies $\|f_{n}-f\|_{B_{\psi}}\rightarrow0~~(n\rightarrow\infty).$  Since $K$ is closed, we have $f\in K$, that means $K$ is compact.
\end{proof}

\section{Main Results}
\begin{thm}
Let $\varphi$ be an analytic self-map of the unit disk $\mathbb{D}$ and $g\in H(\mathbb{D})$, $\psi\in\mathcal{D}$. Then $V_{\varphi}^{g}: AL_{\psi}^{2}(\mathbb{D})\rightarrow B_{\psi}(\mathbb{D})$ is bounded if and only if $$\sup_{z\in\mathbb{D}}e^{\psi(\varphi(z))-2\psi(z)}\sqrt{\Delta\psi(\varphi(z))}|g'(z)|<\infty.$$
\end{thm}
\begin{proof}
By Lemma 2.2, $$e^{\psi(\varphi(z))}\sqrt{\Delta\psi(\varphi(z))}\sim \sqrt{K(\varphi(z),\varphi(z))},$$
therefore, it suffices for us to prove that
$V_{\varphi}^{g}$ is bounded if and only if $$\underset {z\in\mathbb{D}}{\textup{sup}}\sqrt{K(\varphi(z),\varphi(z))}e^{-2\psi(z)}|g'(z)|<\infty.$$

First, assume that $V_{\varphi}^{g}:AL_{\psi}^{2}(\mathbb{D})\rightarrow B_{\psi}(\mathbb{D})$ is bounded. For every $w\in\mathbb{D}$, let $f_{w}(z)=k_{w}(z)=\frac{K(z,w)}{\sqrt{K(w,w)}}.$ It is easy to check that $f_{w}\in AL_{\psi}^{2}(\mathbb{D})$ and  $\|f_{w}\|_{L_{\psi}^{2}}=1$ for any $w\in\mathbb{D}$.

Hence, for a fixed $w\in \mathbb{D}$,
\begin{eqnarray*}
|f_w(\varphi(z))||g'(z)|e^{-2\psi(z)}
&=&|(V_{\varphi}^{g}f_{w})^{'}(z)|e^{-2\psi(z)}\\
&\leq&\|V_{\varphi}^{g}f_{w}\|_{B_{\psi}}\\
&\leq&\|V_{\varphi}^{g}\|\|f_{w}\|_{L_{\psi}^{2}}\\
&=&\|V_{\varphi}^{g}\|< \infty
\end{eqnarray*}
for any $z\in \mathbb{D}.$
Noticing that
$$f_w(\varphi(z))=\frac{K(\varphi(z),w)}{\sqrt{K(w,w)}},~~~~\mbox{for all}~z\in \mathbb{D}, $$
by setting $w=\varphi(z)$, we have
$$f_{\varphi(z)}(\varphi(z))=\frac{K(\varphi(z),\varphi(z))}{\sqrt{K(\varphi(z),\varphi(z))}}=\sqrt{K(\varphi(z),\varphi(z))}.$$
It follows that
 $$\underset {z\in\mathbb{D}}{\textup{sup}}\sqrt{K(\varphi(z),\varphi(z))}|g'(z)|e^{-2\psi(z)}<\infty.$$

Conversely, assume that $$\underset {z\in\mathbb{D}}{\textup{sup}}\sqrt{K(\varphi(z),\varphi(z))}|g'(z)|e^{-2\psi(z)}=M<\infty.$$
For any $f\in AL_{\psi}^{2}(\mathbb{D})$. By Lemma 2.3, we have
\begin{eqnarray*}
\|V_{\varphi}^{g}f\|_{B_{\psi}}&=&|V_{\varphi}^{g}f(0)|+\underset {z\in\mathbb{D}}{\textup{sup}}|(V_{\varphi}^{g}f)'(z)|e^{-2\psi(z)}\\
&=&\underset {z\in\mathbb{D}}{\textup{sup}}|f(\varphi(z))g'(z)|e^{-2\psi(z)}\\
&\leq&\underset {z\in\mathbb{D}}{\textup{sup}}\|f\|_{L_{\psi}^{2}}\sqrt{K(\varphi(z),\varphi(z))}e^{-2\psi(z)}|g'(z)|\\
&=&\|f\|_{L_{\psi}^{2}}M.
\end{eqnarray*}
Therefore, $V_{\varphi}^{g}$ is bounded.
\end{proof}

\begin{thm}
Let $\varphi$ be an analytic self-map of the unit disk $\mathbb{D}$ and $g\in H(\mathbb{D})$, $\psi\in\mathcal{D}$. Then $V_{\varphi}^{g}: AL_{\psi}^{2}(\mathbb{D})\rightarrow B_{\psi,0}(\mathbb{D})$ is bounded if and only if $V_{\varphi}^{g}: AL_{\psi}^{2}(\mathbb{D})\rightarrow B_{\psi}(\mathbb{D})$ is bounded and $$\lim_{|z| \rightarrow 1}e^{-2\psi(z)}|g'(z)|=0.$$
\end{thm}
\begin{proof}
Assume that $V_{\varphi}^{g}: AL_{\psi}^{2}(\mathbb{D})\rightarrow B_{\psi,0}(\mathbb{D})$ is bounded. It is clear that
$V_{\varphi}^{g}: AL_{\psi}^{2}(\mathbb{D})\rightarrow B_{\psi}(\mathbb{D})$ is bounded. Taking $f(z)=1\in AL_{\psi}^{2}(\mathbb{D})$ and $V_{\varphi}^{g}f\in B_{\psi,0}(\mathbb{D})$, then
\begin{eqnarray*}
0=\underset {|z|\rightarrow1}{\textup{lim}}|(V_{\varphi}^{g}f)^{'}(z)|e^{-2\psi(z)}&=&\underset {|z|\rightarrow1}{\textup{lim}}|f(\varphi(z))||g'(z)|e^{-2\psi(z)}\\
&=&\underset {|z|\rightarrow1}{\textup{lim}}|g'(z)|e^{-2\psi(z)}
\end{eqnarray*}

Conversely, suppose that $V_{\varphi}^{g}: AL_{\psi}^{2}(\mathbb{D})\rightarrow B_{\psi}(\mathbb{D})$ is bounded and $\lim_{|z| \rightarrow 1}e^{-2\psi(z)}|g'(z)|=0$. For each polynomial $p(z)$, the following inequality holds
\begin{eqnarray*}
|(V_{\varphi}^{g}p)^{'}(z)|e^{-2\psi(z)}&=&|p(\varphi(z))||g'(z)|e^{-2\psi(z)}\\
&\leq& M_{p}|g'(z)|e^{-2\psi(z)}.
\end{eqnarray*}
where $M_{p}=\underset {z\in\mathbb{D}}{\textup{sup}}|p(z)|$. Since $M_{p}<\infty$ and $\underset {|z|\rightarrow1}{\textup{lim}}|g'(z)|e^{-2\psi(z)}=0$, then
$$\underset {|z|\rightarrow1}{\textup{lim}}|(V_{\varphi}^{g}p)^{'}(z)|e^{-2\psi(z)}=0.$$
That means for each polynomial $p$, $V_{\varphi}^{g}p(z)\in B_{\psi,0}(\mathbb{D})$. Since the set consisting of  polynomials is dense in $AL_{\psi}^{2}(\mathbb{D})$, for every $f\in AL_{\psi}^{2}(\mathbb{D})$, there is a sequence of polynomials $\{p_{k}\}_{k\in\mathbb{N}}$ such that
$$\|f-p_{k}\|_{L_{\psi}^{2}}\rightarrow 0~~~~(k\rightarrow\infty).$$
Hence,$$\|V_{\varphi}^{g}f-V_{\varphi}^{g}p_{k}\|_{B_{\psi}}\leq\|V_{\varphi}^{g}\|\|f-p_{k}\|_{L_{\psi}^{2}}
\rightarrow0(k\rightarrow\infty).$$
Since $V_{\varphi}^{g}p_{k}\in B_{\psi,0}(\mathbb{D})$ and $B_{\psi,0}(\mathbb{D})$ is the the closed subset of $B_{\psi}(\mathbb{D})$,
we have $V_{\varphi}^{g}(AL_{\psi}^{2}(\mathbb{D}))\subset B_{\psi,0}(\mathbb{D})$.\\
Since $V_{\varphi}^{g}:AL_{\psi}^{2}(\mathbb{D})\rightarrow B_{\psi}(\mathbb{D})$ is bounded , we see that $V_{\varphi}^{g}:AL_{\psi}^{2}(\mathbb{D})\rightarrow B_{\psi,0}(\mathbb{D})$ is bounded.
\end{proof}

\begin{thm}
Let $\varphi$ be an analytic self-map of the unit disk $\mathbb{D}$ and $g\in H(\mathbb{D})$, $\psi\in\mathcal{D}$. Then $V_{\varphi}^{g}: AL_{\psi}^{2}(\mathbb{D})\rightarrow B_{\psi}(\mathbb{D})$ is compact if and only if $$\lim_{|\varphi(z)|\rightarrow 1}e^{\psi(\varphi(z))-2\psi(z)}\sqrt{\Delta\psi(\varphi(z))}|g'(z)|=0.$$
\end{thm}
\begin{proof}
By Lemma 2.2, $$e^{\psi(\varphi(z))}\sqrt{\Delta\psi(\varphi(z))}\sim \sqrt{K(\varphi(z),\varphi(z))},$$
therefore, we should only show that $V_{\varphi}^{g}$ is compact if and only if
$$\lim_{|\varphi(z)|\rightarrow 1}\sqrt{K(\varphi(z),\varphi(z))}e^{-2\psi(z)}|g'(z)|=0.$$

First, assume that $\lim_{|\varphi(z)|\rightarrow 1}\sqrt{K(\varphi(z),\varphi(z))}e^{-2\psi(z)}|g'(z)|=0$, then for any $\varepsilon>0$, there is a positive real number $r_{0}\in(0,1)$ such that $$\sqrt{K(\varphi(z),\varphi(z))}e^{-2\psi(z)}|g'(z)|<\varepsilon$$ when $r_{0}<|\varphi(z)|<1$.
Besides, $\sqrt{K(\varphi(z),\varphi(z))}e^{-2\psi(z)}|g'(z)|$ is bounded when $|\varphi(z)|\leq r_{0}$, it is easy to see that
$$\underset {z\in\mathbb{D}}{\textup{sup}}\sqrt{K(\varphi(z),\varphi(z))}e^{-2\psi(z)}|g'(z)|<\infty.$$
By Theorem 3.1,  $V_{\varphi}^{g}:AL_{\psi}^{2}(\mathbb{D})\rightarrow B_{\psi}(\mathbb{D})$ is bounded.

Let $\{f_{k}\}_{k\in\mathbb{N}}$ be a bounded sequence in $AL_{\psi}^{2}(\mathbb{D})$ which uniformly converges to zero on any compact subset of $\mathbb{D}$ as $k\rightarrow\infty$.
Assume that for any $k\in\mathbb{N}$,$\|f_{k}\|\leq C,$ for some positive constant $C$.
Note that $$\underset {|\varphi(z)|\leq r_{0}}{\textup{sup}}|g'(z)|e^{-2\psi(z)}\leq M$$ for some constant $M$.
It follows that
\begin{eqnarray*}
|(V_{\varphi}^{g}f_{k})^{'}(z)|e^{-2\psi(z)}&=&|f_{k}(\varphi(z))||g'(z)|e^{-2\psi(z)}\\
&\leq&\underset {|\varphi(z)|\leq r_{0}}{\textup{sup}}|f_{k}(\varphi(z))||g'(z)|e^{-2\psi(z)}\\
&&+
\underset {|\varphi(z)|> r_{0}}{\textup{sup}}|f_{k}(\varphi(z))||g'(z)|e^{-2\psi(z)}\\
&\leq& M\underset {|\varphi(z)|\leq r_{0}}{\textup{sup}}|f_{k}(\varphi(z))|\\
&&+
\|f_{k}\|_{L_{\psi}^{2}}\underset {|\varphi(z)|> r_{0}}{\textup{sup}}\sqrt{K(\varphi(z),\varphi(z))}|g'(z)|e^{-2\psi(z)}\\
&\leq& M\underset {|\varphi(z)|\leq r_{0}}{\textup{sup}}|f_{k}(\varphi(z))|+\varepsilon\|f_{k}\|_{L_{\psi}^{2}}\\
&\leq& M\underset {|\varphi(z)|\leq r_{0}}{\textup{sup}}|f_{k}(\varphi(z))|+\varepsilon C.
\end{eqnarray*}
Since $\{f_{k}\}_{k\in\mathbb{N}}$ uniformly converges to zero on any compact subset of $\mathbb{D}$, there exists a $ K\in\mathbb{Z}^{+}$ such that if $k>K,$  we have  $ \underset {|\varphi(z)|\leq r_{0}}{\textup{sup}}|f_{k}(\varphi(z))|<\varepsilon. $
Therefore, $\|V_{\varphi}^{g}f_{k}\|_{B_{\psi}}\rightarrow0$ as $k\rightarrow\infty$. By Lemma 2.4, we see that $V_{\varphi}^{g}:AL_{\psi}^{2}(\mathbb{D})\rightarrow B_{\psi}(\mathbb{D})$ is compact.

Conversely, suppose that $V_{\varphi}^{g}:AL_{\psi}^{2}(\mathbb{D})\rightarrow B_{\psi}(\mathbb{D})$ is compact, then it is clear that $V_{\varphi}^{g}:AL_{\psi}^{2}(\mathbb{D})\rightarrow B_{\psi}(\mathbb{D})$ is bounded. Let $\{z_{k}\}_{k\in\mathbb{N}}$ be sequence in $\mathbb{D}$ such that $\underset {k\rightarrow\infty}{\textup{lim}}|\varphi(z_{k})|=1$. Let $$f_{k}(z)=\frac{K(z,\varphi(z_{k}))}{\sqrt{K(\varphi(z_{k}),\varphi(z_{k}))}},$$
then, $f_{k}\in AL_{\psi}^{2}(\mathbb{D})$ and $\|f_{k}\|_{L_{\psi}^{2}}=1$.

Since the set consisting of polynomials is dense in $ AL_{\psi}^{2}(\mathbb{D})$, for any $\varepsilon>0$ and $f\in AL_{\psi}^{2}(\mathbb{D})$,  there exists a polynomial $P_{f,\varepsilon}(z)\in AL_{\psi}^{2}(\mathbb{D})$  such that
$$\|P_{f,\varepsilon}-f\|_{L_{\psi}^{2}}<\frac{\varepsilon}{2}.$$

As
\begin{eqnarray*}
|\langle f_{k},f\rangle|&\leq& |\langle f_{k},f-P_{f,\varepsilon}\rangle|+|\langle P_{f,\varepsilon},f_{k}\rangle|\\
&\leq&\|f_{k}\|_{L_{\psi}^{2}}\|f-P_{f,\varepsilon}\|_{L_{\psi}^{2}}+|\langle P_{f,\varepsilon},f_{k}\rangle|
\end{eqnarray*}
by Lemma 2.2 and Definition 1.1, we have $$\sqrt{K(z,z)}\geq C_{1}\tau(z)^{-1}e^{\psi(z)}\geq\frac{C_{2}e^{\psi(z)}}{1-|z|}$$ for some positive constants $C_{1}$ and $C_{2}$.
Notice that $\lim_{k\rightarrow \infty}|\varphi(z_k)|=0$, we have
$$\langle P_{f,\varepsilon},f_{k}\rangle=\langle P_{f,\varepsilon},\frac{K_{\varphi(z_{k})}}{\|K_{\varphi(z_{k})}\|_{L_{\psi}^{2}}}\rangle
=\frac{1}{\|K_{\varphi(z_{k})}\|_{L_{\psi}^{2}}}P_{f,\varepsilon}(\varphi(z_{k}))\rightarrow 0~~~~
(k\rightarrow\infty)$$
That means that $f_{k}$ weakly converges to zero  as $k\rightarrow\infty$.\\

Because $V_{\varphi}^{g}:AL_{\psi}^{2}(\mathbb{D})\rightarrow B_{\psi}(\mathbb{D})$ is compact, we see that $\|V_{\varphi}^{g}f_{k}\|_{B_{\psi}}\rightarrow 0 ~~~~(k\rightarrow\infty)$. From the following fact
\begin{eqnarray*}
\|V_{\varphi}^{g}f_{k}\|_{B_{\psi}}&\geq&\underset {z\in\mathbb{D}}{\textup{sup}}|f_{k}(\varphi(z))||g'(z)|e^{-2\psi(z)}\\
&\geq&|f_{k}(\varphi(z_{k}))||g'(z_{k})|e^{-2\psi(z_{k})}\\
&=&\sqrt{K(\varphi(z_{k}),\varphi(z_{k}))}|g'(z_{k})|e^{-2\psi(z_{k})}
\end{eqnarray*}
we immediately obtain that $\underset {|\varphi(z)|\rightarrow1}{\textup{lim}}\sqrt{K(\varphi(z),\varphi(z))}|g'(z)|e^{-2\psi(z)}=0$.
The proof is completed.
\end{proof}

\begin{thm}
Let $\varphi$ be an analytic self-map of the unit disk $\mathbb{D}$ and $g\in H(\mathbb{D})$, $\psi\in\mathcal{D}$. Then the operator $V_{\varphi}^{g}: AL_{\psi}^{2}(\mathbb{D})\rightarrow B_{\psi,0}(\mathbb{D})$ is compact if and only if $V_{\varphi}^{g}: AL_{\psi}^{2}(\mathbb{D})\rightarrow B_{\psi,0}(\mathbb{D})$ is bounded and $$\lim_{|z|\rightarrow 1}e^{\psi(\varphi(z))-2\psi(z)}\sqrt{\Delta\psi(\varphi(z))}|g'(z)|=0.$$
\end{thm}
\begin{proof}
At first, we note that $$\lim_{|z|\rightarrow 1}e^{\psi(\varphi(z))-2\psi(z)}\sqrt{\Delta\psi(\varphi(z))}|g'(z)|=0$$ equals to
$$\underset {|z|\rightarrow1}{\textup{lim}}\sqrt{K(\varphi(z),\varphi(z))}e^{-2\psi(z)}|g'(z)|=0.$$

Firstly, we prove the sufficiency . Let $K=\{f:f\in AL_{\psi}^{2}(\mathbb{D}), \|f\|_{L_{\psi}^{2}}\leq1\}$. As $V_{\varphi}^{g}:AL_{\psi}^{2}(\mathbb{D})\rightarrow B_{\psi,0}(\mathbb{D})$ is bounded, $\{V_{\varphi}^{g}f:f\in K\}$ is the bounded closed set of $B_{\psi,0}(\mathbb{D})$. It suffices to show that $\{V_{\varphi}^{g}f:f\in K\}$ is compact in $B_{\psi,0}(\mathbb{D})$. By Lemma 2.5, it is only to prove
$$\underset {|z|\rightarrow1}{\textup{lim}}\underset {\|f\|_{L_{\psi}^{2}}\leq1}{\textup{sup}}e^{-2\psi(z)}|(V_{\varphi}^{g}f)^{'}(z)|=0.$$
By lemma 2.3, for any $ f\in K,$ we have
\begin{eqnarray*}
e^{-2\psi(z)}|(V_{\varphi}^{g}f)^{'}(z)|&=&e^{-2\psi(z)}|f(\varphi(z))||g^{'}(z)|\\
&\leq&\|f\|_{L_{\psi}^{2}}\sqrt{K(\varphi(z),\varphi(z))}e^{-2\psi(z)}|g'(z)|\\
&\leq&\sqrt{K(\varphi(z),\varphi(z))}e^{-2\psi(z)}|g'(z)|.
\end{eqnarray*}
Note that the condition  $\underset {|z|\rightarrow1}{\textup{lim}}\sqrt{K(\varphi(z),\varphi(z))}e^{-2\psi(z)}|g'(z)|=0$, we have
$$\underset {|z|\rightarrow1}{\textup{lim}}\underset {\|f\|_{L_{\psi}^{2}}\leq1}{\textup{sup}}e^{-2\psi(z)}|(V_{\varphi}^{g}f)^{'}(z)|=0.$$
Therefore, the operator $V_{\varphi}^{g}:AL_{\psi}^{2}(\mathbb{D})\rightarrow B_{\psi,0}(\mathbb{D})$ is compact.

Secondly, we will prove the necessity. Suppose $V_{\varphi}^{g}:AL_{\psi}^{2}(\mathbb{D})\rightarrow B_{\psi,0}(\mathbb{D})$ is compact, it is obvious that
$V_{\varphi}^{g}:AL_{\psi}^{2}(\mathbb{D})\rightarrow B_{\psi}(\mathbb{D})$ is compact. By Theorem 3.3, we have
$$\underset {|\varphi(z)|\rightarrow1}{\textup{lim}}\sqrt{K(\varphi(z),\varphi(z))}e^{-2\psi(z)}|g'(z)|=0.$$
That is, for any $\varepsilon>0$, there exists an $r\in(0,1)$, such that $$\qquad\qquad\qquad\qquad\qquad\sqrt{K(\varphi(z),\varphi(z))}e^{-2\psi(z)}|g'(z)|<\varepsilon.\quad\quad\qquad\qquad\qquad(*)$$
when $r<\varphi(z)<1$.\\

Since $V_{\varphi}^{g}: AL_{\psi}^{2}(\mathbb{D})\rightarrow B_{\psi,0}(\mathbb{D})$ is bounded,
by Theorem 3.2, we have
$$\underset {|z|\rightarrow1}{\textup{lim}}e^{-2\psi(z)}|g'(z)|=0.$$
Let $\varepsilon^{'}=\frac{\varepsilon}{C_{r}}$, there exists a positive real number $\sigma>0$, such that $$e^{-2\psi(z)}|g'(z)|<\frac{\varepsilon}{C_{r}}$$
when $\sigma<|z|<1$, where $C_{r}$ is the upper bound of $\sqrt{K(\varphi(z),\varphi(z))}$ when $|\varphi(z)|\leq r$.
Therefore, we have $$\sqrt{K(\varphi(z),\varphi(z))}e^{-2\psi(z)}|g'(z)|
<C_{r}\frac{\varepsilon}{C_{r}}=\varepsilon.$$
when $\sigma<|z|<1$ and $|\varphi(z)|\leq r$.\\
\par
Combines with $(\ast)$, we see that $\sqrt{K(\varphi(z),\varphi(z))}e^{-2\psi(z)}|g'(z)|<\varepsilon$ when $\sigma<|z|<1$.
Therefore, $$\underset {|z|\rightarrow1}{\textup{lim}}\sqrt{K(\varphi(z),\varphi(z))}e^{-2\psi(z)}|g'(z)|=0.$$
The proof is completed.
\end{proof}

\end{section}

\end{document}